\documentclass[11pt,a4paper]{amsart}

\usepackage{
amssymb,amsmath,amscd, amsthm,stmaryrd,
verbatim, mathrsfs,tikz,color, amsfonts,enumerate, esvect,graphicx
}

\usepackage[T1]{fontenc}
\usepackage{lmodern}
\usepackage{mathtools}
\usepackage[a4paper,margin=28mm]{geometry}
\usepackage{microtype}
\usepackage{enumitem}
\usepackage{needspace}

\usepackage[
colorlinks,linkcolor=blue,urlcolor=cyan,citecolor=cyan
]{hyperref}
\hypersetup{pdftitle={Parabolic filtrations of Soergel bimodules},
 pdfsubject={Parabolic filtrations and mixed Kazhdan--Lusztig positivity},
 pdfkeywords={Soergel bimodules, Hecke algebras, hybrid basis, Kazhdan--Lusztig basis, parabolic subgroups}}
\numberwithin{equation}{section}
\newtheorem{theorem}{Theorem}[section]
\newtheorem{proposition}[theorem]{Proposition}
\newtheorem{lemma}[theorem]{Lemma}
\newtheorem{corollary}[theorem]{Corollary}
\theoremstyle{definition}
\newtheorem{definition}[theorem]{Definition}

\theoremstyle{remark}

\newcommand{\cB}{\mathcal B}
\newcommand{\cF}{\mathcal F}
\newcommand{\cH}{\mathcal H}
\newcommand{\kk}{\mathbb R}
\newcommand{\bbZ}{\mathbb Z}
\newcommand{\bbN}{\mathbb N}
\newcommand{\FD}{\cF_\Delta}
\newcommand{\FI}{\cF_\Delta^I}
\newcommand{\Wmin}{{}^IW}
\newcommand{\Acoef}{\bbZ[v,v^{-1}]}
\newcommand{\Homg}{\underline{\operatorname{Hom}}}
\newcommand{\Gr}{\operatorname{Gr}}
\newcommand{\End}{\operatorname{End}}
\newcommand{\Supp}{\operatorname{Supp}}
\newcommand{\Ann}{\operatorname{Ann}}
\newcommand{\Res}{\operatorname{Res}}
\newcommand{\Spec}{\operatorname{Spec}}
\newcommand{\grk}{\operatorname{grk}}
\newcommand{\ch}{\operatorname{ch}}
\newcommand{\id}{\operatorname{id}}
\newcommand{\Ind}{\operatorname{Ind}}
\newcommand{\len}{\ell}
\newcommand{\Dy}[2]{\mathsf D_{#1}(#2)}
\newcommand{\GY}[1]{G^I_{#1}}
\newcommand{\step}[1]{\par\smallskip\noindent\emph{#1}\enspace}
\setlist[enumerate]{label=\textup{(\roman*)},leftmargin=2.2em,
 itemsep=4pt,topsep=5pt,parsep=0pt}
\allowdisplaybreaks[1]

\usepackage{mathtools}
\mathtoolsset{showonlyrefs}

\usepackage{enumitem}
\setlist{itemsep=1pt}

\usepackage{pdfpages} 

\usepackage{todonotes}

\numberwithin{equation}{section}
\allowdisplaybreaks

\title[Parabolic filtrations]{Parabolic filtrations of Soergel bimodules}
\author{Xun Xie}
\author{Qifan Zhou}
\date{}

\begin{document}
\begin{abstract}
Given a Coxeter system $(W,S)$ and a subset $I\subset S$. Let $W_I=\langle I\rangle$ be the standard parabolic subgroup of a Coxeter group $W$, and $\Wmin$ be the set of representatives of minimal length in the coset $W_I\backslash W$. We construct a parabolic version of Soergel's standard filtration, and show that, given any Soergel bimodule, each subquotient of its parabolic filtration is a direct sum of grading shifts of $B_x^I\otimes_R R_y$ with $x\in W_I$ and $y\in\Wmin$. By applying the character theorem of Elias and Williamson, we deduce the positivity of coefficients in the expansion of $C_w$ with respect to the hybrid basis $\{C_xT_y:x\in W_I,y\in\Wmin\}$. This generalize Grojnowski and Haiman's result to any Coxeter group.
\end{abstract}
\maketitle

\section{Introduction}

For Hecke algebra $\cH$ of any Coxeter system $(W,S)$, Kazhdan and Lusztig \cite{KL79} defined the so-called Kazhdan-Lusztig polynomials as the entries of transformation matrix between the Kazhdan-Lusztig basis and the standard basis. In 1980, Kazhdan and Lusztig \cite{KL80} proved that these polynomials have non-negative coefficients for any finite and affine Weyl group. Their result was based on the geometric categorification of the Hecke algebra. Later, Elias and Williamson \cite{EW14} proved the positivity of Kazhdan-Lusztig polynomials for arbitrary Coxeter systems by the meaning of Soergel bimodules.

There are still some well-known positivity problems in Kazhdan-Lusztig theory. For example, Lusztig \cite[Section 3]{Lus85} mentioned that the structure constants of Kazhdan-Lusztig basis of any affine Hecke algebra have positive coefficients, and this is true for all Hecke algebras by \cite{EW14}. Another positivity theorem was given by Grojnowski and Haiman \cite{GH07}. They considered the "hybrid basis" of $\cH$, whose elements are products of some Kazhdan-Lusztig basis and standard basis elements. And they showed that if $W$ is a Coxeter group of a symmetrizable Kac-Moody algebra, the operator of right multiplication by a Kazhdan-Lusztig basis element has positive matrix coefficients with respect to the hybrid basis. Their result generalizes the two above positivity theorems. To prove their main theorem, Grojnowski and Haiman used the geometric categorification of the Hecke algebra, thus $W$ should be the Coxeter group of a symmetrizable Kac-Moody algebra. Following from \cite{Soe07}, we can now categorify any Hecke algebra via Soergel bimodules, so it is natural to ask if one can generalize Grojnowski and Haiman's work to all Hecke algebras.

Let $(W,S)$ be a Coxeter system with $S$ finite, and let $I\subseteq S$.
Denote by $W_I$ the subgroup of $W$ generated by $I$ and by $\Wmin$ the set of representatives of minimal length in the coset $W_I\backslash W$.
Each $w\in W$ has a unique expression $w=x\cdot y$ with $x\in W_I$ and
$y\in\Wmin$, and $\len(w)=\len(x)+\len(y)$.

Let $\cH$ be the Hecke algebra associated to $(W,S)$ over $\Acoef$. $\cH$ admits two bases: the standard basis $\{T_w:w\in W\}$ and the Kazhdan-Lusztig basis $\{C_w:w\in W\}$. For any $I\subset S$, the set $\{C_xT_y:x\in W_I,y\in\Wmin\}$ also forms a basis of $\cH$, which is called the hybrid basis. The definition of hybrid basis was inspired by work of Geck \cite{Gec03}, and many people are interested in the associated positivity problem. As a corollary of Grojnowski and Haiman's result, in the expansion of any $C_w$ with respect to the hybrid basis, all coefficients are polynomials in $\bbN[v,v^{-1}]$. We will generalize this result to any $W$ in this paper.

In \cite{Wil20}, Williamson introduced this positivity theorem and gave some examples. In this paper, we will give a complete proof of this problem. Our proof is mainly based on the theory of Soergel bimodules. Following Soergel's idea in \cite{Soe07}, we construct a parabolic version of standard filtration and prove that the expansion of $C_w$ with respect to the hybrid basis $\{C_xT_y:x\in W_I,y\in\Wmin\}$ corresponds to this parabolic filtration. Thus we can deduce the positivity of coefficients.

\subsection{The parabolic filtration theorem}\label{sec:int_to_main}

Fix the standard real reflection-faithful representation $V$ of $W$ in \cite[Section 2]{Soe07}. Let $R$ be the ring of regular functions on $V$, it admits a $\bbZ$ grading with the dual space $V^*$
in degree $2$. For any graded $R$ bimodule $M$, we also regard it as an $R\otimes_\kk R$ module. We use $[-]$ to denote the grading shift functor, i.e. $M[n]_d=M_{d+n}$.

For $u\in W$, let $R_u$ and $B_u$ be the graded $R$ bimodules defined in \cite{Soe07}.
Let $\cB(W)$ denote the category of Soergel bimodules associated to $W$, and let
$\cB_I$ be its full subcategory generated by the bimodules $B_t$, $t\in I$, under tensor products, finite direct sums, grading shifts and direct summands. We would add a subscript $\cdot^I$ to denote the indecomposable objects in $\cB_I$.

For $u\in W$, put
\[
 \Gr(u)=\{(u\xi,\xi):\xi\in V\},\qquad \Delta_u=R_u[-\len(u)].
\]

For $U\subseteq W$, the notation $\Gamma_U M$ denotes the submodule of $M$ consisting of
elements whose support is contained in $\Gr(U)=\bigcup_{u\in U}\Gr(u)$.

Let $\pi_I:W\to\Wmin$ send $w=x\cdot y$ to $y$, where $x\in W_I$ and $y\in\Wmin$. For $y\in\Wmin$, define
\begin{equation}\label{eq:inverse-set}
    U_y=\pi_I^{-1}(\{z\in\Wmin:z\geq y\}),\qquad
    U_y^+=\pi_I^{-1}(\{z\in\Wmin:z>y\}).
\end{equation}
The submodules $\Gamma_{U_y}M$ with the $y$'s in decreasing Bruhat order give a filtration of $M$, we call it the \textbf{parabolic filtration}. We obtain the following subquotient
\begin{equation}\label{eq:coset-layer}
 \GY{y}(M)=\Gamma_{U_y}M/\Gamma_{U_y^+}M.
\end{equation}
Following from Soergel's construction, if $M$ has a standard filtration, so is $\GY{y}(M)$. For any bimodule $N$, set
\begin{equation}\label{eq:tensor-std-mod}
    \Dy{y}{N}=N\otimes_R\Delta_y.
\end{equation}

The following is our main theorem, in which we describe the subquotients of the parabolic filtration of $B\in\cB(W)$.
\begin{theorem}\label{thm:main}
The following statements hold.
\begin{enumerate}
\item For any $B\in\cB(W)$ and $y\in\Wmin$, there is an object
$N_y\in\cB_I$ such that
\begin{equation}\label{eq:main}
 \GY{y}(B)\cong\Dy{y}{N_y}.
\end{equation}
\item In particular, for $w\in W$, there are unique nonnegative integers $m^w_{x,y,n}$
such that
\begin{equation}\label{eq:multiplicity}
 \GY{y}(B_w)\cong
 \bigoplus_{x\in W_I,\ n\in\bbZ}
 \left(\bigl(B_x^I\otimes_R\Delta_y\bigr)[n]\right)^{\oplus m^w_{x,y,n}}.
\end{equation}
Only finitely many triples $(x,y,n)$ have nonzero multiplicity.
\end{enumerate}
\end{theorem}

Our argument is completely based on the theory of Soergel bimodules, so it applies to any Coxeter system $(W,S)$ and every parabolic subgroup in $W$. The proof begins with the trivial bimodule $R$, which is easy to check. We continue by right tensoring with $B_s$ for some $s\in S$. A double coset $W_IyW_{\{s\}}$ contains either one or two $W_I$-cosets. In the first case, we can compute the parabolic subquotients directly. The latter case is more complicated: there are two factors which form a nontrivial extension by \cite{Soe07}. Proposition \ref{prop:adjacent} shows that this extension splits after restricting to $R\otimes_\kk R^s$. Thus we can consider these two factors separately. By induction, we finally prove that Theorem \ref{thm:main} holds for all Bott-Samelson bimodules, and thus for $\cB(W)$. The proof is given in Section \ref{sec:induction}.

\subsection{Positivity and some applications}

By applying Elias and Williamson's result in \cite{EW14}, we can deduce that
\begin{corollary}\label{cor:positive}
Write
\begin{equation}\label{eq:positive}
 C_w=\sum_{x\in W_I,\ y\in\Wmin}p^I_{x,y,w}(v)C_xH_y.
\end{equation}
Then
\begin{enumerate}
\item $p^I_{x,y,w}(v)=\sum_n m^w_{x,y,n}v^n\in\bbN[v,v^{-1}]$;
\item $p^I_{x,y,w}=0$ unless $xy\leq w$, and $p^I_{x,y,w}=1$ if $xy=w$;
\item $p^I_{x,y,w}\in v\bbN[v]$ if $xy<w$.
\end{enumerate}
\end{corollary}

Using Corollary \ref{cor:positive}, we can deduce the positivity of some other polynomials. For example, we have the positivity of a factorization of matrix of Kazhdan-Lusztig polynomials (see \cite{BMS26}), and of the parabolic Kazhdan-Lusztig polynomials (see \cite{GH07}).

Here is the organization of the paper. Section \ref{sec:preliminary} provides some preliminaries for Hecke algebra and Soergel bimodules. In Section \ref{sec:setup}, we describe objects that admit Soergel's standard filtration in another way. This allows us to define the parabolic filtration. Moreover, we give some splitting criterions for some short exact sequences in Section \ref{sec:setup}, \ref{sec:rank-one} and \ref{sec:splitting}. Based on these conclusions, we prove the main theorem in Section \ref{sec:induction}. Finally, we prove the positivity for the hybrid basis and give some applications in Section \ref{sec:characters}.

\subsection*{AI disclosure}
Before using GPT6, we had worked on this problem for some time. Although the basic idea had already been formed and some progress had been made, certain details remained unresolved. GPT6 helped us overcome these issues, and we have reviewed and verified GPT6’s arguments.

\section{Soergel bimodules}\label{sec:preliminary}
In this section, we recall some basic contents of Hecke algebra and Soergel bimodules. We recommend \cite{EMTW20} for more information.

\subsection{Hecke algebra}

Let $(W,S)$ be a Coxeter system. The Hecke algebra $\cH=\cH(W)$ is an algebra over $\Acoef$. As a free $\Acoef$ module, $\cH$ admits the standard basis $\{T_w:w\in W\}$ satisfying the following relations

$$T_wT_u=T_{wu}\quad if\;\len(wu)=\len(w)+\len(u);$$
$$T_s^2=(v^{-2}-1)T_s+v^{-2}\quad for\;s\in S.$$

Write $H_w=v^{\len(w)}T_w$. There is an automorphism of $\cH$ called the bar involution $\overline{\cdot}$, which is the linear extension of $\overline{v}=v^{-1}$ and $\overline{T_w}=T_{w^{-1}}^{-1}$. Kazhdan and Lusztig \cite{KL79} proved that for each $w\in W$, there is a (unique) bar-invariant element $C_w\in H$ such that
$$C_w=H_w+\sum_{y<w}P_{y,w}(v)H_y,$$
where $P_{y,w}(v)\in v\bbZ[v]$. The set $\{C_w:w\in W\}$ forms a basis of $\cH$ called the Kazhdan-Lusztig basis, and we call the polynomials $P_{y,w}(v)$ Kazhdan-Lusztig polynomials. Note that $C_w$ is denoted by $C_w'$ in \cite{KL79}.

\subsection{Soergel's categorification theorem}

Now we recall some background on Soergel bimodules. Suppose that $V$ is a reflection-faithful representation of $W$. Let $R=\kk[V^*]$ denote the polynomial ring on $V$, we can view $R$ as a $\bbZ$-graded ring with elements in $V^*$ lying in degree 2. Let $\mathcal{R}$ be the category of finitely generated $\bbZ$-graded $R$-bimodules, whose Hom space $\Homg_\mathcal{R}(M,N)$ consists of degree 0 $R$-bimodule homomorphisms for $M,N\in \mathcal{R}$.

The category $\mathcal{R}$ has a monoidal structure with tensor product $\otimes_R$ and identity $R$. We often write $MN$ for the tensor product of $M,N\in \mathcal{R}$. Define the dual of $M\in \mathcal{R}$ by $\mathbb{D}M=\Homg_{-R}(M,R)$, the set of right $R$-module homomorphisms from $M$ to $R$ of all degrees. We can view $\mathbb{D}M$ as an $R$-bimodule via $(rfr')(m)=f(rmr')$, where $r,r'\in R$, $f\in\mathbb{D}M$ and $m\in M$. Note that $\mathbb{D}M$ has obvious $\bbZ$-grading.

Let $[\mathcal{R}]$ be the split Grothendieck group of $\mathcal{R}$, the class of $M$ in $[\mathcal{R}]$ is denoted by $[M]$. The formula $[M]\cdot[N]=[MN]$ makes $[\mathcal{R}]$ a ring. $[\mathcal{R}]$ is a free $\Acoef$-module with $v$ acting via $v[M]=[M[1]]$.

The right action of $W$ on $R$ is written as an exponential natation: $f^u(\xi)=f(u\xi)$. Given a simple reflection $s\in S$, let $R^s$ denote $s$-invariants of $R$. Set $B_s:=R\otimes_{R^s}R[1]$, it is a graded $R$-bimodule with $R$ acting via left and right multiplication. For any expression $\underline{w}=s_{i_1}s_{i_2}\cdots s_{i_l}$, we denote
$$BS(\underline{w})=B_{s_{i_1}}B_{s_{i_2}}\cdots B_{s_{i_l}},$$
which is called the Bott-Samelson bimodule. Clearly, $BS(\underline{w})\cong R\otimes_{R^{s_{i_1}}}R\otimes_{R^{s_{i_2}}}\cdots\otimes_{R^{s_{i_l}}}R[l]$.

\begin{definition}
    The category of \textbf{Soergel bimodules}, denoted by $\cB$, is the smallest full subcategory of $\mathcal{R}$ which contains $R$ and $B_s,s\in S$, and is closed under direct sums, direct summands, tensor product over $R$, and grading shifts.
\end{definition}

In particular, Bott-Samelson bimodules are Soergel bimodules. By \cite[Satz 6.16]{Soe07}, for a reduced expression $\underline{w}$ of $w\in W$, there is a unique (up to isomorphism) indecomposable direct summand $B_w$ of $BS(\underline{w})$ which does not occur as a direct summand of $BS(\underline{y})$ for any shorter expression $\underline{y}$. These bimodules $B_w$ (and their shifts) constitute the set of isomorphism classes of indecomposables in $\cB$. Moreover, $B_w$ is self-dual for each $w\in W$.

Write $[\cB]$ for the split Grothendieck group of $\cB$. We have Soergel's categorification theorem \cite[Theorem 1.10]{Soe07}.
\begin{theorem}\label{cat}
    There is an isomorphism of $\Acoef$-algebras
    $$\varepsilon: H\xrightarrow[]{\sim}[\cB]$$
    such that $\varepsilon(v)=[R(1)]$ and $\varepsilon(T_s+1)=[R\otimes_{R^s}R]$ for $s\in S$.
\end{theorem}

\subsection{Standard bimodules}

There remain some important objects in $\mathcal{R}$ called standard bimodules. We first recall Soergel's construction. For $w\in W$ consider the "twisted" graph
$$\Gr(w)=\{(w\xi,\xi):\xi\in V\}\subset V\times V.$$

View it as a subvariety of $V\times V$. For any subset $A$ of $W$ we define
$$\Gr(A)=\bigcup_{w\in A}\Gr(w)\subset V\times V.$$

We can identify $R\otimes_\kk R$ with the ring of regular functions on $V\times V$. Then $R(A):=R(\Gr(A))$, the ring of regular functions on $\Gr(A)$, is naturally a graded $R$-bimodule and is finitely generated as both left and right $R$-module. When $A=\{x\}$, we write $R_x=R(A)$. One can check that $R_x\cong R$ as left modules, and the right action on $R_x$ is twisted by $x^{-1}$: $m.r=m(r^{x^{-1}})$ for $m\in R_x$ and $r\in R$. We call $R_x\in \mathcal{R}$ the \textbf{standard bimodule}. Write
\[
\Delta_x=R_x[-\len (x)].
\]

Let $\mathcal{R}_{std}$ be the smallest strictly full subcategory of $\mathcal{R}$ which contains $R_x$ for all $x\in W$ and is closed under direct sums and grading shifts. From the definition, we have $R_x\otimes_R R_y\cong R_{xy}$, so $\mathcal{R}_{std}$ is monoidal. Moreover, it categorifies the group algebra $\Acoef[W]$.

The most important application of standard bimodules is constructing the "standard filtration" of some $R$-bimodules. For $M\in\mathcal{R}$, the support of $M$ is $\Supp(M)=\{\mathfrak{p}\in\Spec(R\otimes_\kk R):M_{\mathfrak{p}}\neq0\}$. For $A\subset W$, consider the following subbimodule
\[
\Gamma_A M=\{m\in M:\Supp (m)\subset \Gr(A)\}.
\]
Put $\Gamma_{\geq i} M=\Gamma_{\{x\in W:\len(x)\geq i\}} M$. We say that $M$ has a standard filtration if

\begin{enumerate}
    \item $\Supp(M)\subset \Gr(A)$ for some finite $A\subset W$;

    \item For all $i\in\bbZ$, the subquotient $\Gamma_{\geq i} M/\Gamma_{>i} M$ is isomorphic to a finite direct sum of $R_x[n]$ with $\len(x)=i$ and $n\in\bbZ$.
\end{enumerate}

Denote by $\FD(W)$ the full subcategory of $\mathcal{R}$ consisting of objects which admit a standard filtration (see \cite[Section 5]{Soe07}). In particular, $\cB\subset \FD(W)$. Let $(M:\Delta_x[n])$ be the multiplicity of $\Delta_x[n]$ in $\Gamma_{\geq \len(x)} M/\Gamma_{>\len(x)} M$. We define

\begin{align}
    \ch_\Delta:\FD(W)&\to\cH\\
    M&\mapsto \sum_{x\in W,n\in\bbZ}(M:\Delta_x[n])v^nH_x.
\end{align}

By \cite[Theorem 5.3]{Soe07}, $\ch_\Delta$ defines a left inverse of $\varepsilon$ in Theorem \ref{cat}. Soergel conjectured that $ch_\Delta([B_x])=C_x$ for all $x\in W$. It is equivalent to saying $\End(B_x)=\kk$. Elias and Williamson \cite{EW14} proved this conjecture by the Hodge theory of Soergel bimodules. This indicates the positivity of Kazhdan-Lusztig polynomials. Their work is a significant progress in the theory of Soergel bimodules. 

\section{Ordered filtration and parabolic filtration}\label{sec:setup}

Recall that there exists a real reflection-faithful representation $V$ constructed in \cite[Proposition 2.1]{Soe07}, and denote $R=\kk[V^*]$. For any subgroup $W_I$ of $W$, denote by $R^{W_I}$ the $W_I$ invariants in $R$. 

\begin{lemma}\label{lem:separating}
For every reflection $r\notin W_I$, there is a linear form
$\lambda\in(V^*)^{W_I}$ such that $\lambda^r\neq\lambda$.
\end{lemma}
\begin{proof}
The linearly independence of the simple coroots allows us to choose $\lambda$ with
\[
 \lambda(\alpha_t^\vee)=0\quad for\;t\in I,\qquad
 \lambda(\alpha_s^\vee)>0\quad for\;s\in S\setminus I.
\]
The first condition shows that $\lambda$ is $W_I$-invariant. For any reflection $r\notin W_I$, let $\alpha_r^\vee$ be the corresponding positive coroot. When we write $\alpha_r^\vee$ as a sum of simple coroots, it has positive coefficient for some $\alpha_s^\vee$, $s\in S\setminus I$. Therefore $\lambda(\alpha_r^\vee)>0$, moreover, $\lambda^r\neq \lambda$.
\end{proof}

\subsection{Ordered filtrations}\label{sec:flags}

Put $D=R\otimes_\kk R$. Any $M\in\mathcal{R}$ is also a $D$-module. Let $A$ be the minimal (finite) subset of $W$ such that $\Supp(M)\subset\Gr(A)$, then
\begin{equation}\label{eq:finite-support}
 \Gamma_U M=\Gamma_{U\cap A}M \qquad for\;any\;U\subseteq W.
\end{equation}

Given a poset $A\subset W$, we can rearrange it such that elements bigger in Bruhat order appear earlier. We call such a list the \textbf{decreasing linear extension} of $A$. For example, $\{s_1s_2s_1,s_1s_2,s_2s_1,s_1,s_2,e\}$ is a decreasing linear extension of $\mathcal{S}_3$. The following lemma provides an equivalent description of objects in $\FD(W)$. For a proof, one can see \cite[Lemma 5.8 and Lemma 6.3]{Soe07}.

\begin{lemma}\label{lem:ordered-flags}
Let $A\subset W$ be finite and let $\{w_1,\ldots,w_m\}$ be a decreasing linear extension of $A$.
\begin{enumerate}
\item If $M\in\FD(W)$ satisfies $\Supp(M)\subset\Gr(A)$, then
\[
 M_j=\Gamma_{\{w_1,\ldots,w_j\}}M\qquad(1\leq j\leq m)
\]
is a filtration, and
\[
 M_j/M_{j-1}\cong
 \bigoplus_n\Delta_{w_j}[n]^{\oplus(M:\Delta_{w_j}[n])}.
\]
\item Conversely, suppose that $0=M_0\subset M_1\subset\cdots\subset
M_m=M$ is a filtration and each $M_j/M_{j-1}$ is a direct sum of shifts
of $\Delta_{w_j}$. Then $M\in\FD(W)$.
\end{enumerate}
\end{lemma}
Indeed, by \cite[Lemma 5.8]{Soe07}, if two subquotients $M_{j+1}/M_j$ and $M_j/M_{j-1}$ correspond to incomparable elements in $W$, there is no extension between them. Thus, swapping their order does not change any subquotient. In particular, we can arrange elements in $A$ to be a decreasing linear extension. We call the filtration in Lemma \ref{lem:ordered-flags} the \textbf{ordered filtration}.

\subsection{Subquotients in the parabolic filtration}\label{sec:grouping}

The projection $\pi_I:W\to\Wmin$ (defined in Section \ref{sec:int_to_main}) is order-preserving for Bruhat order (see \cite[Section 2.5]{BB05}). Thus both $U_y$ and $U_y^+$ in \eqref{eq:inverse-set} are upper sets. Recall the subquotient $\GY{y}(M)$ in \eqref{eq:coset-layer}. Let $M\in\FD(W)$ and $A\subset W$ be minimal with $\Supp(M)\subset\Gr(A)$. Lemma~\ref{lem:ordered-flags}
allows us to arrange elements in $A$ in decreasing order, and then within each coset. It follows that $\GY{y}(M)\in\FD(W)$ and
\begin{equation}\label{eq:layer-multiplicities}
 (\GY{y}(M):\Delta_u[n])=
 \begin{cases}
 (M:\Delta_u[n]),&\pi_I(u)=y,\\
 0,&\pi_I(u)\neq y.
 \end{cases}
\end{equation}
More generally, this observation works for any order-preserving map for the same reason. We will apply this to the double coset $W_I\backslash W/W_{\{s\}}$ later.

Recall the definition of $\Dy{y}{N}$ in \eqref{eq:tensor-std-mod}. Let $I\subset S$, given $N\in\FD(W_I)$ and $y\in\Wmin$. $R_y$ is free as left $R$-module, the functor $-\otimes_R R_y$ is exact. Right tensoring with $R_y$ sends the subquotient of ordered filtration of $N$ with label $a\in W_I$ to one with label $ay$. The map $a\mapsto ay$ is order-preserving, Lemma~\ref{lem:ordered-flags}(ii) therefore shows that $\Dy{y}{N}\in\FD(W)$. We will compute its character explicitly in Lemma~\ref{lem:coset-character}.

A short exact sequence
\[
 0\longrightarrow A\longrightarrow E\longrightarrow C\longrightarrow0
\]
of objects in $\FD(W)$ is called \textbf{standard-exact} if
\begin{equation}\label{eq:standard-exact}
 (E:\Delta_u[n])=(A:\Delta_u[n])+(C:\Delta_u[n])
 \quad\text{for all }u,n.
\end{equation}
For example, given $M\in\FD(W)$, the sequence $0\longrightarrow\Gamma_{\geq i}M\longrightarrow M\longrightarrow M/\Gamma_{\geq i}M\longrightarrow0$ is a standard-exact sequence.

\subsection{The Hom splitting criterion}

For $M,N\in\mathcal{R}$, let $\Homg(M,N)$ be the set of $R$-bimodule homomorphisms from $M$ to $N$ of all degrees. It is naturally a graded $R$-bimodule.

Denote by $\grk M$ the graded rank of a graded free right $R$-module $M$. Soergel's Hom formula \cite[Theorem~5.15 and Lemma~6.13]{Soe07}
states that $\Homg(M,N)$ is graded free as right $R$-module for $M\in\FD(W)$ and $N\in\cB$. For fixed $N$, $\grk\Homg(M,N)$ is additive in the multiplicities $(M:\Delta_u[n])$. That is,
\begin{equation}\label{eq:hom-additivity}
 \grk\Homg(M,N)=\sum_{u\in W,\ n\in\bbZ}
 (M:\Delta_u[n])\,\grk\Homg(\Delta_u[n],N).
\end{equation}
This formulation avoids introducing the opposite $\nabla$-filtration, which is not needed here.

\Needspace{8\baselineskip}
\begin{lemma}\label{lem:hom-split}
Let $0\to A\xrightarrow{i}E\to C\to0$ be a standard-exact sequence in
$\FD(W)$. If $A\in\cB$, then the sequence splits in $\mathcal{R}$.
\end{lemma}
\begin{proof}
Apply $\Homg(-,A)$ to the sequence. Left exactness gives
\[
 0\longrightarrow\Homg(C,A)\longrightarrow\Homg(E,A)
 \xrightarrow{i^*}\Homg(A,A).
\]
All three modules are graded free. Standard-exactness \eqref{eq:standard-exact} and
\eqref{eq:hom-additivity} give
\[
 \grk\Homg(E,A)=\grk\Homg(C,A)+\grk\Homg(A,A).
\]
Each homogeneous component is finite dimensional. Left exactness of $\Homg$ identifies the kernel of $i^*$ with $\Homg(C,A)$, so in every degree $d$,
\[
 \dim_{\kk}\operatorname{im}(i^*)_d
 =\dim_{\kk}\Homg(E,A)_d-\dim_{\kk}\Homg(C,A)_d
 =\dim_{\kk}\Homg(A,A)_d.
\]
Thus $i^*$ is surjective in each homogeneous component. By lifting $\id_A\in\Homg(A,A)_0$ to $\Homg(E,A)_0$, we obtain a section of $i:A\to E$. Thus the sequence splits.
\end{proof}

The assumption of standard-exactness is essential. The lemma also holds for $A\in\cB_I$ and $E,C\in\FD(W_I)$.

\section{$R$-bimodules supported in $\Gr(W_IyW_{\{s\}})$}\label{sec:rank-one}

\subsection{The double coset $W_IyW_{\{s\}}$}

\begin{lemma}\label{lem:product-order}
Let $y\in\Wmin$ and $s\in S$.
\begin{enumerate}
\item Either $ys\in\Wmin$, or $ys=ty$ for some $t\in I$.
In the latter case $y<ys$ and $W_IyW_{\{s\}}=W_Iy$.
\item Suppose that $z=ys\in\Wmin$ and $y<z$. Then
\[
 W_IyW_{\{s\}}=W_Iy\sqcup W_Iz,\qquad r=ysy^{-1}\notin W_I.
\]
The map
\[
 W_I\times\{0,1\}\longrightarrow W_Iy\sqcup W_Iz,
 \qquad(a,\varepsilon)\longmapsto ays^\varepsilon
\]
is an isomorphism of posets, where the order of $W_I\times\{0,1\}$ is induced by that of $W_I$ and $0<1$.
\end{enumerate}
\end{lemma}
\begin{proof}
(i) Suppose that $ys\notin\Wmin$, and choose $t\in I$ with $tys<ys$.
If $ys<y$, then
\[
 \len(ty)\leq\len(tys)+1=\len(ys)<\len(y),
\]
contrary to $y\in\Wmin$. Thus $ys>y$. By the exchange property of Coxeter groups (\cite[Theorem 1.4.3]{BB05}), left multiplication by $t$ removes a letter in $ys$. If this letter appears in $y$, we again have $ty<y$, a contradiction. So the removed letter should be $s$, and $ys=ty$. The coset equality follows.

(ii) One can check the coset equality directly. If $r=ysy^{-1}\in W_I$, then $z=ry\in W_Iy$, and $y,z$ are in the same coset, but this is impossible. Hence $r\notin W_I$.

Now we prove the poset isomorphism. First notice that for $a,b\in W_I$,
\[
 ay\leq by\ \Longleftrightarrow\ a\leq b,
 \qquad az\leq bz\ \Longleftrightarrow\ a\leq b.
\]
Then suppose $ay\leq bz=bys$. The lifting property (\cite[Proposition 2.2.7]{BB05}) shows that $ay\leq by$, moreover $a\leq b$. Conversely, if $a\leq b$, then $ay\leq by\leq bys$. Finally, the case $az\leq by$ is impossible since $\pi_I$ is order-preserving. Thus our map is an isomorphism of posets.
\end{proof}

We also record the projection on the double-coset. It is
\begin{equation}\label{eq:double-projection}
 d_s(w)=\min\{\pi_I(w),\pi_I(w)s\},
\end{equation}
where the minimum is taken in Bruhat order. Lemma~\ref{lem:product-order}
shows that this element lies in $\Wmin$ and is the shortest representative in $W_IwW_{\{s\}}$. The projection $d_s$ is also order-preserving, and its fiber is stable under right multiplication by $s$.

In the situation of part~(ii), the $R$-bimodule $\Gamma_{d_s^{-1}(\{x:x\geq y\})}M/\Gamma_{d_s^{-1}(\{x:x>y\})}M$ has an ordered filtration whose subquotients correspond to elements in $W_IyW_{\{s\}}$. In particular, We can order them as
\begin{equation}\label{eq:two-label-order}
 \{az,ay\},\qquad a\in W_I,
\end{equation}
where the $a$'s occur in a decreasing linear extension and $az$ always precedes $ay$.

\subsection{Extensions between $R_u$ and $R_{us}$}

\begin{lemma}\label{lem:graph-ext}
Let $u\in W$ and $s\in S$. Any extension between finite direct sums $\oplus_{n\in\bbZ}R_u[n]^{\oplus k_n}$ and $\oplus_{m\in\bbZ}R_{us}[m]^{\oplus k_m'}$ in $\mathcal{R}$ splits when the right action of $R$ is restricted to $R^s$.
\end{lemma}
\begin{proof}
Since $s\in S$ fixes a hyperplane in $V$, graphs of $u$ and $us$ meet in codimension one. Recall that $D=R\otimes_\kk R$. By 
\cite[Lemma~5.8]{Soe07}, $\operatorname{Ext}^1_D(R_u,R_{us})$ is free of rank one over $R(\Gr(u)\cap\Gr(us))$. The generating class is of the form
\begin{equation}\label{eq:rank-one-ext}
    0\longrightarrow R_{us}[-2]\longrightarrow R(\Gr(u)\cup\Gr(us))\longrightarrow R_u\longrightarrow0.
\end{equation}

We first show that this generating extension splits after restriction. Left tensoring by $R_{u^{-1}}$ preserves the exactness of \eqref{eq:rank-one-ext}, while the middle term becomes $R(e,s)=R\otimes_{R^s}R$ (see \cite[Remark 4.3]{Soe07}). In the new exact sequence
\begin{equation}\label{eq:new-ses}
    0\longrightarrow R_s[-2]\longrightarrow R\otimes_{R^s}R\overset{\mu}{\longrightarrow} R_e\longrightarrow0,
\end{equation}
the rightmost map is given by $\mu(f\otimes g)=fg$, then the map $h\mapsto h\otimes1$ defines the section of $\mu$ after right restriction to $R^s$. Indeed, if $r\in R^s$, $hr\otimes1=h\otimes r$ in $R\otimes_{R^s}R$. Therefore, \eqref{eq:new-ses} splits after restriction. Moreover, \eqref{eq:rank-one-ext} also splits since we can obtain it by applying an exact functor to \eqref{eq:new-ses}. The statement is true if we take finite direct sums and grading shifts since $\operatorname{Ext}^1$ is additive.
\end{proof}

\subsection{Right translation functor}

Define
\[
 \Theta_s(M)=M\otimes_R B_s
 =\bigl(\Res_{R\text{-}R^s}M\otimes_{R^s}R\bigr)[1].
\]
The functor $\Theta_s$ is exact because $R$ is free of rank two over $R^s$. It preserves $\FD(W)$ by \cite[Proposition~5.7]{Soe07}.

\Needspace{14\baselineskip}
\begin{lemma}\label{lem:support-translation}
Let $M\in\FD(W)$ and $U\subseteq W$ satisfy $Us=\{us:u\in U\}=U$ for some $s\in S$.
\begin{enumerate}
\item The inclusion $\Gamma_U M\hookrightarrow M$ induces an isomorphism
\begin{equation}\label{eq:support-translation}
 \Theta_s(\Gamma_U M)\xrightarrow{\ \sim\ }\Gamma_U(\Theta_sM).
\end{equation}
\item If $U'\subseteq U$ also satisfies $U's=U'$, then
\begin{equation}\label{eq:quotient-translation}
 \Theta_s\bigl(\Gamma_U M/\Gamma_{U'}M\bigr)
 \cong\Gamma_U(\Theta_sM)/\Gamma_{U'}(\Theta_sM).
\end{equation}
\end{enumerate}
\end{lemma}
\begin{proof}
We prove (i) by identifying the two sides as submodules of certain $D$-module. We first reconstruct $\Theta_sM$. Put
\[
 A=R\otimes_{\kk}R^s\subseteq D=R\otimes_{\kk}R.
\]
We view $M$ as an $A$-module. There is a graded $D$-module isomorphism
\begin{equation}\label{eq:theta-base-change}
 \varphi:M\otimes_A D\xrightarrow{\ \sim\ }M\otimes_{R^s}R,
 \qquad m\otimes(f\otimes g)\longmapsto fm\otimes g.
\end{equation}
Here the $D$-action on $M\otimes_A D$ is the product with the second factor, and on $M\otimes_{R^s}R$, $D$ acts by the usual left and right multiplication. Notice that $f\otimes g=(f\otimes1)\cdot(1\otimes g)$ and $f\otimes1\in A$, it is easy to check that $\varphi$ is a well-defined homomorphism of $D$-modules. The inverse of $\varphi$ is given by $m\otimes g\mapsto m\otimes(1\otimes g)$. Thus $\varphi$ is an isomorphism, and $\Theta_sM\cong (M\otimes_A D)[1]$.

Then we explain the action of $\Gamma_U$ on both $M$ and $\Theta_sM$. If we set $R=\kk[\alpha_s,x_2,\dots,x_{\dim V}]$ where $s(\alpha_s)=-\alpha_s$ and $s$ fixes the other coordinates, then $R^s=\kk[\alpha_s^2,x_2,\cdots,x_{\dim V}]$. The inclusion $A\hookrightarrow D$ induces the quotient map $p:V\times V=\Spec D\to\Spec A= V\times (V/\langle s\rangle)$. For any $(u\xi,\xi)\in\Gr(U)$, the fiber of $p(u\xi,\xi)$ consists of two points: $(u\xi,\xi)$ and $(u\xi,s\xi)$. Notice that $(u\xi,s\xi)=(us(s\xi),s\xi)\in\Gr(us)$. Since $Us=U$, we have $p^{-1}\bigl(p(\Gr(U)\bigr)=\Gr(U)$.

Let $F_0\subset W$ be minimal with $\Supp(M)\subset \Gr(F_0)$, and $F=F_0\cup F_0s$. Then $F$ is finite, and \eqref{eq:finite-support} allows us to replace $U$ by the finite set $U\cap F$ in $\Gamma_U M$. Moreover, $\Theta_s M$ has support contained in $\Gr(F)$, the same replacement applies to $\Gamma_U(\Theta_sM)$.

Let $X=\Gr(U\cap F)$. The fact $(U\cap F)s=U\cap F$ implies
\begin{equation}\label{eq:inverse-support}
    p^{-1}\bigl(p(X)\bigr)=X.
\end{equation}
By our construction, $p$ is finite, $Z=p(X)$ is therefore closed. Take its defining ideal $J\subset A$. Equation \eqref{eq:inverse-support} says that $X$ has defining ideal $JD\subset D$.

For an $A$-module $N$, write
\[
H_J^0(N)=\{n\in N:J^kn=0\text{ for some }k\in\bbN\}.
\]
Consider the cyclic $D$-module $Dm\cong D/\Ann_D(m)$ with $m\in M$, the following are equivalent:
\[
 \begin{aligned}
 \Supp_D(Dm)\subseteq X
 &\ \Longleftrightarrow\ JD\subseteq\sqrt{\Ann_D(m)}\\
 &\ \Longleftrightarrow\ J^km=0\text{ for some }k\in\bbN.
 \end{aligned}
\]
The notation $\sqrt{\cdot}$ means the radical ideal. Thus
\begin{equation}\label{eq:support-torsion-A}
\begin{split}
    \Gamma_U M&=\{m\in M:\Supp(m)\subset \Gr(U)\}\\
    &=\{m\in M:\Supp(Am)\subset Z\}\\
    &=\{m\in M:J^km=0\text{ for some }k
    \in\bbN\}=H_J^0(M),
\end{split}
\end{equation}
and similarly
\begin{equation}\label{eq:support-torsion-D}
    \Gamma_U(M\otimes_A D)=H_{JD}^0(M\otimes_A D).
\end{equation}
Finally, we show that $\Theta_s$ and $\Gamma_U$ commute. We put $\beta_s=1\otimes\alpha_s\in D$, then
\begin{equation}\label{D-decomposition}
    R=R^s\oplus R^s\alpha_s,\qquad D=A\oplus A\beta_s.
\end{equation}
Every element in $M\otimes_A D$ can be written uniquely as
\[
m\otimes d=m_0\otimes1+m_1\otimes\beta_s.
\]
It follows that if $m\otimes d\neq0$,
\[
(JD)^k(m\otimes d)=0\Longleftrightarrow J^km_0=J^km_1=0.
\]
This equivalence and \eqref{D-decomposition} induce an isomorphism of $D$-modules
\begin{equation}\label{eq:torsion-base-change}
 H_{JD}^0(M\otimes_A D)=H_J^0(M)\otimes_A D.
\end{equation}
Now combining \eqref{eq:theta-base-change}, \eqref{eq:support-torsion-A}, \eqref{eq:support-torsion-D} and \eqref{eq:torsion-base-change} gives the isomorphism in (i).

Apply the exact functor $\Theta_s$ to
\[
 0\longrightarrow\Gamma_{U'}M\longrightarrow\Gamma_U M
 \longrightarrow\Gamma_U M/\Gamma_{U'}M\longrightarrow0.
\]
The assertion in (ii) is a direct corollary of part (i).
\end{proof}

\section{Extension between modules supported in adjacent coset graphs}\label{sec:splitting}

In this section, we will give another criterion whether a short exact sequence in $\mathcal{R}$ splits after restriction. This is essential for our later discussion.

\begin{proposition}\label{prop:adjacent}
Suppose that $y,z=ys\in\Wmin$ and $y<z$. Let $E\in\FD(W)$ fit into a short exact sequence
\begin{equation}\label{eq:adjacent-extension}
 0\longrightarrow A_1\otimes_R R_z\longrightarrow E
 \longrightarrow A_0\otimes_R R_y\longrightarrow0,
 \qquad A_0,A_1\in\cB_I.
\end{equation}
Then this sequence splits after
restricting the right action of $R$ to $R^s$.
\end{proposition}
\begin{proof}

To prove this proposition, we will define a new $R$-bimodule structure which is compatible with $R-R^s$-action on these bimodules, apply Lemma~\ref{lem:hom-split} to this new action, and finally return to the restricted sequence.

Note that localization preserves exactness, the support of $E$ is contained in the graph of double coset $W_IyW_{\{s\}}=W_Iz\sqcup W_Iy$ by Lemma \ref{lem:product-order}. In \eqref{eq:adjacent-extension}, $A_0\otimes_R R_y$ has no nonzero element supported in $\Gr(W_Iz)$. The kernel of $E\to A_0\otimes_R R_y$ is therefore exactly $\Gamma_{W_Iz}E$. Thus \eqref{eq:adjacent-extension} is standard-exact.

Tensor on the right by $R_{y^{-1}}$ and put
\[
 \widetilde E=E\otimes_R R_{y^{-1}},\qquad r=ysy^{-1},\qquad Q=R^r.
\]
We have a short exact sequence
\begin{equation}\label{eq:untwisted-extension}
 0\longrightarrow A_1\otimes_R R_r\longrightarrow\widetilde E
 \longrightarrow A_0\longrightarrow0.
\end{equation}
The support of $\widetilde E$ is contained in $\Gr(W_I\cup W_Ir)$. We claim that
\begin{equation}\label{eq:splitting-equivalence}
    \eqref{eq:adjacent-extension} \text{ splits over }(R,R^s)\text{ if and only if }\eqref{eq:untwisted-extension}\text{ splits over }(R,Q).
\end{equation}
Indeed, suppose that we can find an $(R,R^s)$-linear section $\varphi:E\to A_1\otimes_R R_z$. Write $\widetilde\varphi=\varphi\otimes \id_{R_{y^{-1}}}:\widetilde E\to A_1\otimes_R R_r$. For any $q\in Q$, $q^y\in R^s$. Then we can check that for $e\otimes c\in \widetilde E$,
\[
\widetilde\varphi(e\otimes c).q=\varphi(e)\otimes cq^y=\varphi(e).q^y\otimes c=\varphi(e.q^y\otimes c)=\varphi\bigl((e\otimes c).q\bigr).
\]
Thus $\widetilde\varphi$ gives an $(R,Q)$-linear section, and the equivalence of two splitting assertions follows.

Now we construct a new action on $E_0=\Res_{R-Q}\widetilde E$ to make it an $R$-bimodule. Lemma \ref{lem:product-order} gives $r\notin W_I$. Applying Lemma \ref{lem:separating}, we can choose $\lambda\in(V^*)^{W_I}$ with $\lambda^r\neq\lambda$. Since $r=ysy^{-1}$ is a reflection, we have $R=Q\oplus Q\lambda$. More precisely,
\[
R\cong Q[X]/\langle X^2-(\lambda+\lambda^r)X+\lambda\lambda^r\rangle.
\]
Define the right action of $R$ on $E_0$ by retaining the $Q$-action and prescribing
\begin{equation}\label{eq:new-action}
 m\cdot_{\mathrm{new}}\lambda=\lambda m.
\end{equation}
We must check that this action is well-defined. Let $L_f$ and $\rho_f$ denote the original left and right multiplications by $f\in R$ on $\widetilde E$. Note that the left and right actions of $\lambda\in R^{W_I}$ coincide on objects in $\cB_I$. Therefore
\[
 L_\lambda=\rho_\lambda\quad\text{on }A_0,
 \qquad
 L_\lambda=\rho_{\lambda^r}\quad\text{on }A_1\otimes_R R_r.
\]
It follows that
\[
 (L_\lambda-\rho_\lambda)\widetilde E\subseteq A_1\otimes_R R_r,
 \qquad
 (L_\lambda-\rho_{\lambda^r})(A_1\otimes_R R_r)=0.
\]
Composing these two operators we have
\[
(L_\lambda-\rho_{\lambda^r})(L_\lambda-\rho_\lambda)
  =L_\lambda^2-\rho_{(\lambda+\lambda^r)} L_\lambda+\rho_{(\lambda\lambda
  ^r)}=0.
\]
This is the quadratic relation that kills $\lambda$. Since $R$ is commutative, \eqref{eq:new-action} actually defines an $R$-bimodule structure on $E_0$. We denote it by $E^\sharp$.

Apply our new action to each term in \eqref{eq:untwisted-extension}, $A_0$ does not change because $\lambda$ is $W_I$-invariant, while $A_1\otimes_R R_r$ is isomorphic to $A_1$ because $q\in Q$ commutes with $R_r$. Thus we obtain a sequence relative to the new action
\begin{equation}\label{eq:sharp-extension}
 0\longrightarrow A_1\longrightarrow E^\sharp
 \longrightarrow A_0\longrightarrow0.
\end{equation}
A quick check (we omit here) shows that \eqref{eq:sharp-extension} is also exact. 

We have seen $\Supp(E)\subset \Gr(W_Iz\sqcup W_Iy)$. Lemma \ref{lem:ordered-flags} allows us to arrange the elements in the cosets as $\{az,ay\}$ with the $a$'s in decreasing Bruhat order. Then the subquotients are of the form $\oplus_{n\in\bbZ}\Delta_{az}[n]^{\oplus k_n}$ or $\oplus_{m\in\bbZ}\Delta_{ay}[m]^{\oplus k'_m}$. By Lemma \ref{lem:graph-ext}, any extension between such two $R$-bimodules splits after restricting the right $R$-action to $R^s$. Tensoring with $R_{y^{-1}}$ makes the subquotients the finite direct sum of grading shifts of $\Delta_{ar}$ and $\Delta_a$ respectively. For the same reason with \eqref{eq:splitting-equivalence}, any extension between these two subquotients also splits as $(R,Q)$-bimodules.

Any $(R,Q)$-linear splitting map $f$ in an extension is preserved by our new action because
\[
f(m\cdot_{\mathrm{new}}\lambda)
 =f(\lambda m)=\lambda f(m)
 =f(m)\cdot_{\mathrm{new}}\lambda.
\]
Thus the above extensions remain split under the new action. Moreover, the new action identifies $R_{ar}$ and $R_a$. To see this, take $h\in R$ and $q\in Q$, $q$ acts on $R_{ar}$ and $R_a$ respectively as
\[
h.q=hq^{ra^{-1}}=hq^{a^{-1}},\qquad h.q=hq^{a^{-1}}.
\]
The new right action of $\lambda$ is original left multiplication. Thus $R_{ar}\cong R_a$ under the new action.
It follows that $E^\sharp$ has an ordered filtration with each subquotient of the form $\oplus_{n\in\bbZ}\Delta_a[n]^{\oplus l_n}$, and the $a$'s occur in a decreasing linear extension. Lemma \ref{lem:ordered-flags}(ii) gives
\begin{equation}\label{eq:sharp-standard}
 E^\sharp\in\FD(W_I).
\end{equation}
To apply the splitting criterion in Lemma \ref{lem:hom-split}, we need to prove the standard-exactness of \eqref{eq:sharp-extension}. The split extension between $R_{ay}$ (resp. $R_a$) and $R_{az}$ (resp. $R_{ar}$) shows that
\begin{equation}
    \begin{split}
    (E^\sharp:R_a[n])&=(E_0:R_{ar}[n])+(E_0:R_a[n])\\
    &=(\widetilde E:R_{ar}[n])+(\widetilde E:R_a[n])\\
    &=(E:R_{az}[n])+(E:R_{ay}[n]).
    \end{split}
\end{equation}
Use the relation $\Delta_x=R_x[-\len(x)]$ and the length additivity we obtain
\[
(E^\sharp:\Delta_a[n])=(E:\Delta_{az}[n+\len(z)])+(E:\Delta_{ay}[n+\len(y)]).
\]
By the discussion at the beginning of the proof,
\[
(E:\Delta_{az}[n+\len(z)])=(A_1\otimes_R R_z:\Delta_{az}[n+\len(z)])=(A_1:\Delta_a[n]),
\]
\[
(E:\Delta_{ay}[n+\len(y)])=(A_0\otimes_R R_y:\Delta_{ay}[n+\len(y)])=(A_0:\Delta_a[n]).
\]
Therefore,
\begin{equation}\label{eq:sharp-additive}
 (E^\sharp:\Delta_a[n])
 =(A_1:\Delta_a[n])+(A_0:\Delta_a[n])
 \qquad(a\in W_I,\ n\in\bbZ).
\end{equation}
That is, \eqref{eq:sharp-extension} is standard-exact.

Then Lemma \ref{lem:hom-split} indicates that \eqref{eq:sharp-extension} splits in $\mathcal{R}$. Then we are given a section $E^\sharp\to A_1$. Forgetting the new action of $\lambda$ provides an $(R,Q)$-linear section $\widetilde E\to A_1\otimes_R R_r$ of \eqref{eq:untwisted-extension}. Thus \eqref{eq:untwisted-extension} splits over $(R,Q)$. By the equivalence \eqref{eq:splitting-equivalence}, the sequence \eqref{eq:adjacent-extension} splits over $(R,R^s)$.
\end{proof}

\section{Parabolic subquotients of Soergel bimodules}\label{sec:induction}

Let $\FI$ be the full subcategory of $\FD(W)$ consisting of the modules $M$ satisfying
\[
\GY{y}(M)\cong \Dy{y}{N_y} \text{ with }N_y\in\cB_I,\text{ for all }y\in \Wmin.
\]
The object $R$ belongs to $\FI$. The following lemma is about the structure of $\FI$.
\begin{lemma}\label{FI-closed}
    The category $\FI$ is closed under finite direct sums, grading shifts and direct summands.
\end{lemma}
\begin{proof}
    From the definition of $\GY{y}$, we can deduce that this functor is additive. Recall that $\FD(W)$ is closed under direct summands (\cite[Remark 5.5]{Soe07}). If $M=M'\oplus M''$, then
    \[
    \GY{y}(M)=\GY{y}(M')\oplus\GY{y}(M'').
    \]
    If $M',M''\in\FI$, so does $M$ because tensor product commutes with direct sum. Conversely, if $M\in\FI$, then $\GY{y}(M)\otimes_R R_{y^{-1}}[\len(y)]\in\cB_I$, all of its direct summands lie in $\cB_I$. Thus, $M',M''\in\FI$.
\end{proof}

For $M\in\FI$, from the proof we can take
\[
 N_y=\bigl(\GY{y}(M)\otimes_R R_{y^{-1}}\bigr)[\len(y)]\in\cB_I.
\]

\begin{proposition}\label{prop:stable}
Let $M\in\FI$ and $s\in S$.
\begin{enumerate}
\item If $y\in\Wmin$ and $ys=ty$ for some $t\in I$, then
\begin{equation}\label{eq:one-coset}
 \GY{y}(\Theta_sM)\cong\Dy{y}{N_y\otimes_R B_t}.
\end{equation}
\item If $y,z=ys\in\Wmin$ and $y<z$, then
\begin{align}
 \GY{y}(\Theta_sM)&\cong\Dy{y}{N_y[1]\oplus N_z},
 \label{eq:transition-lower}\\
 \GY{z}(\Theta_sM)&\cong\Dy{z}{N_y\oplus N_z[-1]}.
 \label{eq:transition-upper}
\end{align}
\item Consequently, $\Theta_s(\FI)\subseteq\FI$.
\end{enumerate}
\end{proposition}
\begin{proof}
Consider the order-preserving map $d_s$ in \eqref{eq:double-projection}. Put
\begin{align}
U_{y,s}=d_s^{-1}\bigl(\{u:u\text{ minimal in }W_IuW_{\{s\}},u\geq y\}\bigr),\\
U_{y,s}^+=d_s^{-1}\bigl(\{u:u\text{ minimal in }W_IuW_{\{s\}},u>y\}\bigr).    
\end{align}
Following from Lemma \ref{lem:ordered-flags} and the discussion in Section \ref{sec:grouping}, $M$ has a filtration whose subquotients are $\Gamma_{U_{y,s}}M/\Gamma_{U_{y,s}^+}M$. $U_{y,s}$ is stable under right multiplication by $s$. Lemma \ref{lem:support-translation}(ii) therefore allows us to commute taking such subquotients and the action of $\Theta_s$. Let us consider the two cases in Lemma \ref{lem:product-order} separately.

\step{(i) A double coset containing one $W_I$-coset.}
Here $ys=ty$ with $t\in I$, and $W_IyW_{\{s\}}=W_Iy$. It follows that
\[
\GY{y}(\Theta_sM)\cong\Theta_s\bigl(\GY{y}(M)\bigr).
\]
For $M\in\FI$, we have $\GY{y}(M)\cong\Dy{y}{N_y}$. Since $ys=ty$, the map $a\mapsto a^{y^{-1}}$ defines an isomorphism $R^s\overset{\sim}{\to}R^t$. Then we have an isomorphism of $R$-bimodules
\[
R_y\otimes_{R^s}R\overset{\sim}{\to}R\otimes_{R^t}R_y,\qquad f\otimes g\mapsto f\otimes g^{y^{-1}}.
\]
Consequently,
\[
R_y\otimes_R B_s\cong B_t\otimes_R R_y.
\]
This yields
\[
\Theta_s\bigl(\GY{y}(M)\bigr)=\Dy{y}{N_y}\otimes_RB_s=N_y\otimes_R\Delta_y\otimes_RB_s=N_y\otimes_RB_t\otimes_R\Delta_y=\Dy{y}{N_y\otimes_RB_t},
\]
which proves \eqref{eq:one-coset}.

\step{(ii) A double coset containing two $W_I$-cosets.}
In this case, $y,z=ys\in\Wmin$, with $y<z$. Write $E=\Gamma_{U_{y,s}}M/\Gamma_{U_{y,s}^+}M$ for the subquotient supported in $W_IyW_{\{s\}}$. Then $\Theta_sE=\Gamma_{U_{y,s}}(\Theta_sM)/\Gamma_{U_{y,s}^+}(\Theta_sM)$, and $E\in\FD(W)$ has a canonical exact sequence
\begin{equation}\label{eq:two-coset-support}
 0\longrightarrow\Dy{z}{N_z}\longrightarrow E
 \longrightarrow\Dy{y}{N_y}\longrightarrow0.
\end{equation}
Here we again use Lemma \ref{lem:ordered-flags}, and the fact that $M\in\FI$. Now apply Proposition \ref{prop:adjacent} with $A_1=N_z[-\len(z)]$ and $A_0=N_y[-\len(y)]$. It gives a split exact sequence after restriction of right action to $R^s$. Extension of scalars back to $R$ is exactly the action of $\Theta_s$ on \eqref{eq:two-coset-support} (up to a shift $[1]$), and $R$ is of finite rank $2$ over $R^s$, thus we obtain
\begin{equation}\label{eq:translated-split}
 \Theta_s E\cong
 \Theta_s\Dy{y}{N_y}\oplus\Theta_s\Dy{z}{N_z}.
\end{equation}
$B_s$ has two exact sequences given by its filtrations:
\begin{align*}
 0&\longrightarrow R_s[-1]\longrightarrow B_s
       \longrightarrow R[1]\longrightarrow0,\\
 0&\longrightarrow R[-1]\longrightarrow B_s
       \longrightarrow R_s[1]\longrightarrow0.
\end{align*}
Tensor on the left by $\Dy{y}{N_y}$ and $\Dy{z}{N_z}$ respectively, these sequences are still exact:
\begin{align}
 0&\longrightarrow\Dy{z}{N_y}
 \longrightarrow\Theta_s\Dy{y}{N_y}
 \longrightarrow\Dy{y}{N_y}[1]\longrightarrow0,
 \label{eq:up}\\
 0&\longrightarrow\Dy{z}{N_z}[-1]
 \longrightarrow\Theta_s\Dy{z}{N_z}
 \longrightarrow\Dy{y}{N_z}\longrightarrow0.
 \label{eq:down}
\end{align}
In each sequence the first term has support in $\Gr(W_Iz)$, while the last term has support in $\Gr(W_Iy)$. Equations \eqref{eq:translated-split}, \eqref{eq:up} and \eqref{eq:down} show that $\GY{y}(\Theta_sM)\cong\Dy{y}{N_y[1]\oplus N_z}$ and $\GY{z}(\Theta_sM)\cong\Dy{z}{N_y\oplus N_z[-1]}$, that is what we need.

Finally, part (iii) follows from (i) and (ii) obviously.
\end{proof}

Now we can prove our main theorem.
\begin{proof}[Proof of Theorem~\ref{thm:main}]
(i) Proposition \ref{prop:stable} (iii) shows that tensor product with $B_s$ preserves the category $\FI$. Starting with $R\in\FI$, all Bott-Samelson bimodules belong to $\FI$ by induction. Lemma \ref{FI-closed} indicates that $\cB(W)\subset\FI$. This is \eqref{eq:main}.

(ii) Applying \eqref{eq:main} to $B_w$ gives $N_y\in\cB_I$. By the Krull-Schmidt theorem and \cite[Satz 6.16]{Soe07}, we can decompose $N_y$ uniquely to a direct sum of indecomposables $B_x^I[n]$ in $\cB_I$. Then tensoring with $\Delta_y$ gives \eqref{eq:multiplicity}.

Since $\Supp(B_w)$ is always contained in a graph of finite set, and $N_y$ has finitely many direct summands, only finitely many triples $(x,y,n)$ occur.
\end{proof}

\section{The positivity theorem}\label{sec:characters}

\subsection{Positivity with respect to hybrid basis}

\begin{lemma}\label{lem:coset-character}
For $N\in\FD(W_I)$ and $y\in\Wmin$, we have
\begin{equation}\label{eq:character-product}
 \ch_\Delta\Dy{y}{N}=\ch_\Delta(N)H_y.
\end{equation}
\end{lemma}
\begin{proof}
For any decreasing linear extension of $W_I$, right multiplication with $y\in\Wmin$ does not change its order. Moreover, right tensor product with $\Delta_y$ is exact and sends $\Delta_a$ to $\Delta_{ay}$. By Lemma \ref{lem:ordered-flags} we have
\[
(N:\Delta_a[n])=(\Dy{y}{N}:\Delta_{ay}[n]).
\]
Using the relation $H_aH_y=H_{ay}$ then \eqref{eq:character-product} follows.
\end{proof}

Then we can prove the positivity of coefficients in the expansion of $C_w$ with respect to the hybrid basis $\{C_xH_y\}$. 

\begin{proof}[Proof of Corollary~\ref{cor:positive}]
(i) The character theorem \cite{EW14} gives
$\ch_\Delta(B_w)=C_w$ and $\ch_\Delta(B_x^I)=C_x$.
By \eqref{eq:layer-multiplicities}, $\ch_\Delta(B_w)=\sum_{y\in\Wmin}\ch_\Delta\bigl(\GY{y}(M)\bigr)$. Applying
Lemma~\ref{lem:coset-character} to Theorem~\ref{thm:main}(ii) gives
\[
 C_w=\sum_{x\in W_I,\ y\in\Wmin,\ n\in\bbZ}
       m^w_{x,y,n}v^n C_xH_y.
\]
The uniqueness of expansion identifies
$p^I_{x,y,w}=\sum_n m^w_{x,y,n}v^n$, which proves the positivity.

(ii)--(iii) Note that both the Kazhdan-Lusztig basis and the hybrid basis are unitriangular over the standard basis. More precisely, for $x\in W_I$
and $y\in\Wmin$,
\[
 C_xH_y=H_{xy}+\sum_{a<x}P_{a,x}(v)H_{ay},
\]
here $P_{a,x}(v)$ is the Kazhdan-Lusztig basis. Then we can deduce (ii) and (iii) from the properties of Kazhdan-Lusztig polynomials.
\end{proof}

We would like to end with some further corollaries to Corollary \ref{cor:positive}.

\subsection{Factorization of Kazhdan-Lusztig polynomials}

Given a Coxeter system $(W,S)$ of a symmetrizable Kac-Moody algebra, Bhattacharya, Mishra and Srivastava \cite{BMS26} proved that the matrix of Kazhdan-Lusztig polynomials factorizes into a product of some matrices, each of which has entries as polynomials with non-negative coefficients. Here is their main idea. For any $I\subset J\subset S$, the Hecke algebra $\cH$ has two corresponding hybrid bases: $\{C_xT_y:x\in W_I,y\in\Wmin\}$ and $\{C_uT_w:u\in W_J,w\in{}^JW\}$. They proved that each basis element $C_uT_w$ expands positively in $\{C_xT_y:x\in W_I,y\in\Wmin\}$. Then the matrix of Kazhdan-Lusztig polynomials is the product of transformation matrices between these hybrid bases, and their entries have non-negative coefficients. Their main result is based on \cite{GH07}, so we can obtain a more general conclusion.

\begin{corollary}
    Let $(W,S)$ be a Coxeter system. If $I\subset J\subset S$, then for any $u\in W_J$, $w\in{}^JW$, $C_uT_w$ expands positively in $\{C_xT_y:x\in W_I,y\in\Wmin\}$:
    \[
    C_uT_w=\sum_{x\in W_I,y\in\Wmin}h^{I,J}_{x,y,u,w}(q)C_xT_y
    \]
    with $h^{I,J}_{x,y,u,w}(q)\in\bbN[v,v^{-1}]$.
\end{corollary}

\begin{proof}
    Denote by $\langle C_xT_y\rangle C_uT_w$ the coefficient of $C_xT_y$ in the expansion of $C_uT_w$. For $y\in\Wmin$, let $y=y_J\cdot{}^Jy$ be the unique decomposition with $y_J\in W_J,{}^Jy\in{}^JW$. According to the proof of \cite[Theorem 5.1]{BMS26}, we have the following equality.
    \[
    \langle C_xT_y\rangle C_uT_w=\langle C_xT_{y_J}\rangle C_u\delta_{^Jy,w}.
    \]
    Then the positivity follows by Corollary \ref{cor:positive}.
\end{proof}

\subsection{Parabolic Kazhdan-Lusztig polynomials}
Another application of Corollary \ref{cor:positive} is the positivity of parabolic Kazhdan-Lusztig polynomials. These polynomials were defined by Deodhar in \cite{Deo87}, and we recall Soergel's normalization in \cite{Soe97} here. For $I\subset S$, we denote by $\cH_I$ the subalgebra of $\cH$ generated by $\{T_w:w\in W_I\}$. Let $\varepsilon(T_w)=(-1)^{\len(w)}$ be the sign representation of $\cH_I$. The induced representation $\Ind_{\cH_I}^\cH(\varepsilon)$ is an one-dimensional representation $\cH e_I^-$, where $e_I^-$ satisfies
\[
(T_s+1)e_I^-=0\qquad \text{for }s\in I.
\]
$\cH e_I^-$ has two bases: $\{T_we_I^-:w\in W^I\}$ and $\{C_we_I^-:w\in W^I\}$, where $W^I$ is the set of representatives of minimal length in $W/W_I$. The parabolic Kazhdan-Lusztig polynomials $P_{y,w}^{I-}(v)$ are defined to be the coefficients in the expansion
\[
C_we_I^-=\sum_{y\in W^I}P_{y,w}^{I-}(v)T_ye_I^-.
\]
We have $T_yC_xe_I^-=\delta_{e,x}T_ye_I^-$ for $y\in W^I,x\in W_I$. Hence $P_{y,w}^{I-}(v)$ is the coefficient of $T_yC_e$ in the expansion of $C_w$ (with respect to the hybrid basis) for $y\in W^I$ and $w\in W$. In particular, from an opposite version of Corollary \ref{cor:positive} (see \cite[Corollary 3.9]{GH07}), we have the following.
\begin{corollary}
    The parabolic Kazhdan-Lusztig polynomials $P_{y,w}^{I-}(v)$ have non-negative coefficients.
\end{corollary}

\bibliographystyle{alpha}
\bibliography{S1}

@book{BB05,
    doi = {},
    url = {},
    issn = {},
    year = {2005},
    pages = {},
    title = {Combinatorics of Coxeter Groups},
    author = {Anders Bj\"orner and Francesco Brenti},
    journal = {},
    publisher = {Springer, New York}
}

@article{EW14,
    doi = {10.4007/annals.2014.180.3.6},
    url = {http://dx.doi.org/10.4007/annals.2014.180.3.6},
    issn = {},
    year = {2014},
    pages = {1089-1136},
    title = {The Hodge theory of Soergel bimodules},
    author = {Ben Elias and Geordie Williamson},
    journal = {Annals of Mathematics},
    publisher = {}
}

@article{GH07,
    doi = {},
    url = {https://math.berkeley.edu/~mhaiman/ftp/llt-positivity/new-version.pdf},
    issn = {},
    year = {2007},
    pages = {},
    title = {Affine Hecke algebras and positivity of LLT and Macdonald polynomials},
    author = {Iain Grojnowski and Mark Haiman},
    journal = {},
    publisher = {}
}

@article{Soe07,
    doi = {10.1017/S1474748007000023},
    url = {http://dx.doi.org/10.1017/S1474748007000023},
    issn = {},
    year = {2007},
    pages = {501-525},
    title = {Kazhdan--Lusztig-Polynome und unzerlegbare Bimoduln \"uber Polynomringen},
    author = {Wolfgang Soergel},
    journal = {Journal of the Institute of Mathematics of Jussieu},
    publisher = {}
}

@article{BMS26,
    doi = {},
    url = {},
    issn = {},
    year = {2026},
    pages = {},
    title = {ON FACTORIZATION OF MATRIX OF
KAZHDAN-LUSZTIG POLYNOMIALS},
    author = {Aritra Bhattacharya and Ashish Mishra and Shraddha Srivastava},
    note = {\href{https://arxiv.org/abs/2602.19508}{https://arxiv.org/abs/2602.19508}}}

@article{Deo87,
    doi = {},
    url = {},
    issn = {},
    year = {1987},
    pages = {483-506},
    title = {On some geometric aspects of Bruhat orderings. II. The parabolic analogue of Kazhdan-Lusztig polynomials},
    author = {Vinay V. Deodhar},
    journal = {J. Algebra},
    publisher = {}}

@article{Soe97,
    doi = {},
    url = {},
    issn = {},
    year = {1997},
    pages = {83-114},
    title = {Kazhdan-Lusztig polynomials and a combinatoric[s] for tilting modules.},
    author = {Wolfgang Soergel},
    journal = {Represent. Theory},
    publisher = {}}

@article{KL79,
    doi = {},
    url = {https://doi.org/10.1007/BF01390031},
    issn = {},
    year = {1979},
    volume = {53},
    pages = {165-184},
    title = {Representations of Coxeter groups and Hecke algebras},
    author = {David Kazhdan and George Lusztig},
    journal = { Invent. Math.},
    publisher = {}}

@inproceedings{KL80, title={Schubert varieties and Poincar'e duality}, author={David Kazhdan and George Lusztig}, year={1980}, url={https://api.semanticscholar.org/CorpusID:117614624} }

@article{Lus85,
    doi = {10.2969/aspm/00610255},
    url = {https://doi.org/10.2969/aspm/00610255},
    issn = {0920-1971},
    year = {1985},
    pages = {255-287},
    title = {Cells in Affine Weyl Groups},
    author = {George Lusztig},
    journal = {Advanced Studies in Pure Mathematics},
    publisher = {Mathematical Society of Japan}}

@book{EMTW20,
  title={Introduction to Soergel bimodules},
  author={Ben Elias and Shotaro Makisumi and Ulrich Thiel and Geordie Williamson},
  isbn={},
  lccn={},
  series={RSME Springer Series},
  url={},
  year={2020},
  publisher={Springer}
}

@article{Wil20,
  author  = {Geordie Williamson},
  title   = {Modular representations and reflection subgroups},
  journal = {Mathematische Zeitschrift},
  volume  = {296},
  pages   = {1--36},
  year    = {2020},
  doi     = {10.1007/s00209-020-02487-2},
  eprint  = {2001.04569},
  archivePrefix = {arXiv},
  primaryClass  = {math.RT}
}

@article{Gec03,
  author  = {Meinolf Geck},
  title   = {On the Induction of Kazhdan-Lusztig Cells},
  journal = {Bulletin of the London Mathematical Society},
  volume  = {35},
  number  = {5},
  pages   = {608--614},
  year    = {2003},
  doi     = {10.1112/S0024609303002236},
  mrnumber = {1989489}
}

\end{document}